\documentclass[12pt]{amsart}
\usepackage{verbatim,amscd,amssymb}
\usepackage{verbatim,amssymb,graphicx}
\usepackage[active]{srcltx}

\newtheorem{thm}{Theorem}[section]
\newtheorem{prop}[thm]{Proposition}
\newtheorem{lem}[thm]{Lemma}

\newtheorem{cor}[thm]{Corollary}

\newtheorem*{thmA}{Theorem A}
\newtheorem*{thmB}{Theorem B}
\newtheorem*{thmC}{Theorem C}

\theoremstyle{definition}
\newtheorem{dfn}[thm]{Definition}

\def\Ac{\mathcal{A}}  \def\Cc{\mathcal{C}}
\def\C{\mathbb{C}}  \def\Fc{\mathcal{F}}
 \def\al{\alpha}
  \def\Mc{\mathcal{M}}
\def\Pc{\mathcal{P}}  \def\Oc{\mathcal{O}}
  \def\Tc{\mathcal{T}}

\def\R{\mathbb{R}}

\def\Uc{\mathcal{U}} \def\Vc{\mathcal{V}} \def\Wc{\mathcal{W}}

\def\Z{\mathbb{Z}}

\renewcommand\emptyset{\varnothing}
\newcommand{\sm}{\setminus}

\def\eps{\varepsilon}

\def\al{\alpha}
\def\be{\beta}

\def\om{\omega}
\def\la{\lambda}
\def\si{\sigma}
\def\vp{\varphi}
\def\ol{\overline}

\def\thu{\mathrm{TH}}

\def\cu{\mathrm{CU}}
\def\uc{\mathbb{S}}
\def\bd{\mathrm{Bd}}

\def\le{\leqslant}
\def\ge{\geqslant}
\def\0{\emptyset}
\def\disk{\mathbb{D}}
\def\cdisk{\ol{\mathbb{D}}}
\def\phd{\mathrm{PHD}}

\begin{document}
\date{November 8, 2014; revised July 21, 2015; next revision November 30, 2015}
\title[Quadratic-like dynamics of cubic polynomials]
{Quadratic-like dynamics of cubic polynomials}

\author[A.~Blokh]{Alexander~Blokh}

\author[L.~Oversteegen]{Lex Oversteegen}

\author[R.~Ptacek]{Ross~Ptacek}

\author[V.~Timorin]{Vladlen~Timorin}

\address[Alexander~Blokh, Lex~Oversteegen and Ross~Ptacek]
{Department of Mathematics\\ University of Alabama at Birmingham\\
Birmingham, AL 35294}

\address[Vladlen~Timorin]
{Faculty of Mathematics\\
Laboratory of Algebraic Geometry and its Applications\\
Higher School of Economics\\
Vavilova St. 7, 112312 Moscow, Russia }

\address[Vladlen~Timorin]
{Independent University of Moscow\\
Bolshoy Vlasyevskiy Pereulok 11, 119002 Moscow, Russia}

\email[Alexander~Blokh]{ablokh@math.uab.edu}
\email[Lex~Oversteegen]{overstee@math.uab.edu}
\email[Ross~Ptacek]{rptacek@uab.edu}
\email[Vladlen~Timorin]{vtimorin@hse.ru}




\begin{abstract}
A small perturbation of a quadratic polynomial $f$ with a
non-repelling fixed point gives a polynomial $g$ with an attracting
fixed point and a Jordan curve Julia set, on which $g$ acts like
angle doubling. However, there are cubic polynomials with a
non-repelling fixed point, for which no perturbation results into a
polynomial with Jordan curve Julia set. Motivated by the study of
the closure of the Cubic Principal Hyperbolic Domain, we describe
such polynomials in terms of their quadratic-like restrictions.
\end{abstract}

\maketitle

\section{Introduction}
In this paper, we study topological dynamics of complex cubic
polynomials. We denote the \emph{Julia set} of a polynomial $f$ by
$J(f)$ and the \emph{filled Julia set} of $f$ by $K(f)$. Let us
recall classical facts about quadratic polynomials.
The Mandelbrot set $\Mc_2$, perhaps the most well-known mathematical
set outside of the mathematical community, can be defined as the set
of all complex numbers $c$ such that the sequence
$$
c,\quad c^2+c,\quad (c^2+c)^2+c,\dots
$$
is bounded. The numbers $c$ label polynomials $z^2+c$. Every
quadratic polynomial can be reduced to this form by an affine
coordinate change.

By definition, $c\in \Mc_2$ if the orbit of $0$ under $z\mapsto
z^2+c$ is bounded. What is so special about the point 0? It is the
only critical point of the polynomial $z^2+c$ in $\C$. A critical
point of a complex polynomial has a meaning in the realm of
topological dynamics. Namely, this is a point that does not have a
neighborhood, on which the map is one-to-one.
Generally, the behavior of critical orbits
to a large extent determines
the dynamics of other orbits.
For example, by a
classical theorem of Fatou and Julia, $c\in \Mc_2$ if and only if
the filled Julia set of $z^2+c$
$$
K(z^2+c)=\{z\in\C\,|\, z,\ z^2+c,\ (z^2+c)^2+c,\
\dots\not\to\infty\}
$$
is connected. If $c\not\in \Mc_2$, then the set $K(z^2+c)$ is a
Cantor set.

The Mandelbrot set has a complicated fractal shape. Yet one can see
many components of the interior of $\Mc_2$ bounded by real analytic
curves (in fact, ovals of real algebraic curves). The central part
of the Mandelbrot set, the \emph{Principal Hyperbolic Domain}
$\phd_2$, is bounded by a cardioid (a curve, whose shape resembles
that of a heart). This cardioid is called the \emph{Main Cardioid}.
By definition, the Principal Hyperbolic Domain $\phd_2$ consists of
all parameter values $c$ such that the polynomial $z^2+c$ is
hyperbolic, and the set $K(z^2+c)$ is a Jordan disk (a polynomial of
any degree is said to be {\em hyperbolic} if the orbits of all its
critical points converge to attracting cycles). Equivalently,
$c\in\phd_2$ if and only if $z^2+c$ has an attracting fixed point.

The closure of $\phd_2$ consists of all parameter values $c$ such
that $z^2+c$ has a non-repelling fixed point. As follows from the
Douady--Hubbard landing theorem \cite{DH,Hu} and
Pommerenke-Levin-Yoccoz inequality \cite{Hu,lev91,pom86,pet93}, the
Mandelbrot set itself can be thought of as the union of the main
cardioid and \emph{limbs} (connected components of
$\Mc_2\sm\ol{\phd}_2$) parameterized by reduced rational fractions
$p/q\in (0,1)$. This motivates our study of cubic analogs of
$\phd_2$ started in \cite{bopt14} and continued in \cite{bopt14a}.
We begin our discussion by describing some results of these two
papers.

Complex numbers $c$ are in one-to-one correspondence with affine
conjugacy classes of quadratic polynomials (throughout we call
affine conjugacy classes of polynomials \emph{classes} of
polynomials). Thus, a higher-degree analog of the set $\Mc_2$ is the
\emph{degree $d$ connectedness locus} $\Mc_d$, i.e., the set of
classes of degree $d$ polynomials $f$, all of whose critical points
do not escape (equivalently, whose Julia set $J(f)=\bd(K(f))$ is
connected). The {\em Principal Hyperbolic Domain} $\phd_d$ of
$\Mc_d$ is the set of classes of hyperbolic degree $d$ polynomials
with Jordan curve Julia sets; the class $[f]$ of a degree $d$
polynomial $f$ belongs to $\phd_d$ if and only if all critical
points of $f$ are in the immediate attracting basin of the same
attracting (or super-attracting) fixed point. In \cite{bopt14} we
describe properties of polynomials $f$ such that $[f]\in \ol\phd_d$.

\begin{thm}[Theorem A \cite{bopt14}]\label{t:prophd}
If $[f]\in \ol{\phd}_d$, then $f$ has a fixed non-repelling
point, no repelling periodic cutpoints in $J(f)$,  and all its
non-repelling periodic points, except at most one fixed point, have
multiplier 1.
\end{thm}

Actually, in \cite{bopt14} we claim that all non-repelling periodic
\emph{cutpoints} in the Julia set $J(f)$, except perhaps $0$, have
multiplier 1; still, literally repeating the same arguments one
proves the version of \cite[Theorem A]{bopt14} given by
Theorem~\ref{t:prophd}. This motivates Definition~\ref{d:cubio}, in
which we define a special set $\cu$ such that $\ol\phd_3\subset
\cu$.

\begin{dfn}[\cite{bopt14}]\label{d:cubio}
Let $\cu$ be the family of classes of cubic polynomials $f$ with
connected $J(f)$ such that $f$ has a non-repelling fixed point, no
repelling periodic cutpoints in $J(f)$,  and all its non-repelling
periodic points, except at most one fixed point, have multiplier 1. The
family $\cu$ is called the \emph{Main Cubioid}. 
\end{dfn}

Let $\Fc$ be the space of polynomials
$$
f_{\lambda,b}(z)=\lambda z+b z^2+z^3,\quad \lambda\in \C,\quad b\in \C
$$
parameterized by pairs $(\lambda,b)$ of complex numbers. An affine
change of variables reduces any cubic polynomial $f$ to the form
$f_{\lambda,b}$. The point $0$ is fixed for every $f_{\lambda,b}\in
\Fc$. The set of all polynomials $f\in\Fc$ such that $0$ is
non-repelling for $f$ is denoted by $\Fc_{nr}$. Define the
\emph{$\la$-slice} $\Fc_\lambda$ of $\Fc$ as the space of all
polynomials $g\in\Fc$ with $g'(0)=\lambda$. The space $\Fc$ maps onto
the space of classes of all cubic polynomials with a fixed point of
multiplier $\lambda$ as a finite branched covering equivalent to the
map $b\mapsto a=b^2$, i.e., classes of polynomials
$f_{\la,b}\in\Fc_\la$ are in one-to-one correspondence with the values
of $a$. Thus, if we talk about, say, points $[f]$ of $\Mc_3$, then it
suffices to take $f\in\Fc_\lambda$ for some $\lambda$.

Let $J(f)$ be connected. In \cite{lyu83, MSS}, the notion of
\emph{$J$-stability} was introduced for any holomorphic family of
rational functions: a map is \emph{$J$-stable} if its Julia set admits
an equivariant holomorphic motion over a neighborhood of the map in the
given family. We say that $f\in \Fc_\la$ is \emph{stable} if it is
$J$-stable with respect to $\Fc_\lambda$ with $\lambda=f'(0)$,
otherwise we say that $f$ is \emph{unstable}. The set $\Fc^{st}_\la$ of
all stable polynomials $f\in \Fc_\la$ is an open subset of $\Fc_\la$. A
component of $\Fc^{st}_\la$ is called a \emph{$(\lambda$-$)$stable
component} or a \emph{domain of $(\la$-$)$stability}. It is easy to see
that, given $\la$, the polynomial $f_{\lambda,b}$ has a disconnected
Julia set if $|b|$ is sufficiently big. Hence, if $f=f_{\lambda,b}$ is
stable and $J(f)$ is connected, then its domain of stability is
bounded. For any subset $\Uc\subset\Fc$, we let $[\Uc]$ stand for the
set of classes $[f]$ of all polynomials $f\in\Uc$. If $|\la|\le 1$,
then we write $\Pc_\la$ for the set of all polynomials $f\in\Fc_\la$
such that $[f]\in\ol\phd_3$.

\begin{thm}[\cite{bopt14}]\label{t:extend}
Suppose that $\Uc$ is a bounded stable component in $\Fc_\la$,
$|\la|\le 1$, such that $[\bd(\Uc)]\subset\ol\phd_3$.
Then $[\Uc]\subset \cu$.
Thus, if $\Tc_\la$ is the union of $\Pc_\la$ and all $\lambda$-stable
components, whose boundaries are contained in $\Pc_\la$, then
$[\Tc_\la]\subset \cu$.
\end{thm}

For a compactum $X\subset \C$, let the \emph{topological hull
$\thu(X)$} of $X$ be the complement of the unbounded component of
$\C\setminus X$. Lemma~\ref{l:extend-th-i} follows from
Corollary~\ref{c:bdd-stab} and gives an equivalent description of
$\Tc_\la$.


\begin{lem}\label{l:extend-th-i}
Any component $\Wc$ of\, $\thu(\Pc_\la)\sm\Pc_\la$ consists of stable
maps. Moreover, $\Tc_\la=\thu(\Pc_\la)$ and $\thu(\Pc_\la)\subset \cu$.
\end{lem}

In \cite{bopt14a}, we study properties of components of
$\thu(\Pc_\la)\sm \Pc_\la$, where $|\la|\le 1$.
Note that these are the same as bounded components of $\Fc_\la\sm\Pc_\la$.
Let $\Ac$ be the set of all maps $f_{\la, b}$ with $|\la|<1$.
For each such map $f$, let $A(f)$ be the basin
of immediate attraction of $0$. In 
Section \ref{s:notinphd}, we show that if $f\in \Fc_\la\sm \Pc_\la$,
then $f$ has two distinct critical points. A critical point $c$ of
such an $f$ is said to be \emph{principal} if there is a
neighborhood $\Uc$ of $f$ in $\Fc$ and a holomorphic function
$\omega_1:\Uc\to\C$ with $c=\omega_1(f)$, and such that, for every
$g\in\Uc\cap\Ac$, the point $\omega_1(g)$ is a unique critical point
of $g$ contained in $A(g)$ (the uniqueness of $c$ follows from the
fact that $g\notin \ol{\phd_3}$). By Section \ref{s:notinphd}, the
point $\om_1(f)$ is well-defined; denote the other critical point of
$f$ by $\om_2(f)$.

\begin{dfn}[\cite{bopt14a}]\label{d:sieg-cap}
Let $\Uc$ be a component of $\thu(\Pc_\la)\sm\Pc_\la$. If, for every
$f\in\Uc$, the point $0$ belongs to a Siegel domain $U$ of $f$, and
there exists an eventual pullback $V$ of $U$ such that $f|_V$ is
two-to-one, we call $\Uc$ a component of \emph{Siegel capture} type.
If, for every $f\in\Uc$, the map $f$ has connected Julia set $J(f)$ of
positive Lebesgue measure, carries a measurable invariant line field
and is such that $\om_2(f)\in J(f)$, then $\Uc$ is said to be of
\emph{queer} type.
\end{dfn}

Theorem~\ref{t:phd-info} relies upon the tools developed in
\cite{bopt14}.

\begin{thm}[\cite{bopt14a}]\label{t:phd-info}
Suppose that $|\la|\le 1$. Then any component of
$\thu(\Pc_\la)\sm\Pc_\la$ is either of Siegel capture type or of
queer type.
\end{thm}

We do not know whether components of Siegel capture type or of queer type exist.

\section{Main results}\label{s:main}

Let us begin by making a few remarks. If we perturb a cubic
polynomial $f$ with a non-repelling fixed point to a polynomial $g$
with an attracting fixed point, then $g$ restricted to the basin of
attraction $A(g)$ of that point is either two-to-one or
three-to-one. Here, we study polynomials $f$ with a non-repelling
fixed point such that $[f]\not\in\ol\phd_3$, i.e., \emph{for all
cubic polynomials $g$ sufficiently close to $f$, if $g$ has a fixed
attracting point, then $g$ is two-to-one on its basin of immediate
attraction}. While interesting by itself, this together with
\cite{bopt14} also allows us to learn more about the structure of
$\ol\phd_3$.

Next we need a few classic definitions and a major result due to
Douady and Hubbard \cite{DH85b}.

\begin{dfn}[\cite{DH85b}]\label{d:pl}
A \emph{polynomial-like map (of degree $k$)} is a holomorphic map
$f: U \to V$ where every point in $V$ has exactly $d$ preimages in
$U$ (counted with multiplicities), $U, V$ are open sets isomorphic
to the unit disc, and $\ol{U}\subset V$. A polynomial-like map of
degree $2$ is called a \emph{quadratic-like} map. The \emph{filled
Julia set} $K(f)$ of $f$ is the set of points in $U$ that never
leave $U$ under iteration. The \emph{Julia set} $J(f)$ of $f$ is the
boundary of $K(f)$. Two quadratic-like maps $f:U\to V$ and $g:U'\to
V'$ are \emph{hybrid equivalent} if there is a quasi-conformal map
$\vp:U\to U'$ conjugating $f$ to $g$ such that $\vp$ is conformal
almost everywhere on $K(f)$. The map $\vp$ is called a
\emph{straightening map}.
\end{dfn}

It is easy to see that under hybrid equivalence repelling periodic
points cannot correspond to non-repelling periodic points.

The following major result is due to Douady and Hubbard; we state it
only in the quadratic case.

\begin{thm}[\cite{DH85b}]\label{t:straigth}
Let $f : U \to V$ be a quadratic-like map. Then $f$ is hybrid
equivalent to a quadratic polynomial $P$. Moreover, if $K(f)$ is
connected, then $P$ is unique up to (global) conjugation by an
affine map.
\end{thm}

Say that a cubic polynomial $f\in\Fc$ is \emph{immediately
renormalizable} if there are Jordan domains $U^*\ni 0$ and $V^*$
such that $f^*=f:U^*\to V^*$ is a quadratic-like map (we will use
the notation $f^*$ at several occasions in the future when we talk
about immediately renormalizable maps). If $f\in \Fc_{nr}$ is
immediately renormalizable, then the quadratic-like Julia set
$J(f^*)=J^*$ is connected. Indeed, $f^*$ is hybrid equivalent to a
quadratic polynomial $g^*$. Since $0\in J(f^*)$ is a non-repelling
$f$-fixed point, it corresponds to a non-repelling fixed point of
$g^*$.
Hence, $J(g^*)$ and $J(f^*)$ are connected, and $g^*=z^2+c$ with
$c\in\ol\phd_2$.

Note that, if $[f]\in\ol\phd_3$, then
$f$ is not immediately renormalizable. Indeed, 
if $f:U^*\to V^*$ is quadratic-like, then $g:U^*\to V^*$ is
quadratic-like for all $g$ sufficiently close to $f$. If $f, g\in
\Fc_{nr}$ and $0\in U^*$, then the quadratic-like Julia sets of both
$f$ and $g$ are connected. Thus, $[g]$ cannot belong to $\phd_3$.
If $f\in\Fc_{nr}$ but $[f]\not\in\ol\phd_3$, then $f$ is said to be
\emph{potentially renormalizable}. Clearly, the set of all
potentially renormalizable polynomials is open in $\Fc_{nr}$. A
connected component of the set of potentially renormalizable
polynomials in $\Fc_\lambda$ is called a \emph{potentially
renormalizable component}. For any $f\in \Fc_\la\sm\Pc_\la$, let
$\Wc_f$ be a potentially renormalizable component containing $f$; the set
$\Wc_f$ is open in $\Fc_\la$, and the set of all potentially renormalizable
polynomials is open in $\Fc$.

We will now discuss the main results of this paper. In
Section~\ref{s:notinphd} we introduce special sets $Z(f)$ and $X(f)$
and study them using \emph{holomorphic motion}. We define a
\emph{countable set $Z(f)$ of iterated preimages of the principal
critical point $\omega_1(f)$} as follows: a point $z\in\C$ belongs
to $Z(f)$ if there exists an open convex neighborhood $\Uc_z$ of $f$
in $\Fc$ and a holomorphic function $\zeta:\Uc_z\to\C$ so that (1)
$\zeta(f)=z$, (2) $g^{\circ n}(\zeta(g))=\omega_1(g)$ for all
$g\in\Uc_z$ and for some $n\ge 0$ independent of $g$, and (3)
$\zeta(g)\in A(g)$ for all $g\in\Uc_z\cap\Ac$. Then we study
properties of $Z(f)$, define sets $Z_n(f)$ as the subsets of $Z(f)$
consisting of all preimages of $\omega_1(f)$ mapped to $\omega_1(f)$
in $n$ steps, and define the set $X(f)$ as the limit of the sets
$Z_n(f)$, i.e., $X(f)=\bigcap_{m\ge 0}\ol{\bigcup_{n\ge m}Z_n(f)} $.
The main result of Section~\ref{s:notinphd} is Theorem A.

\begin{thmA}
Suppose that $f\in \Wc_f$, where $\Wc_f$ is a potentially
renormalizable component of $\Fc_\la\sm \Pc_\la$. Then there is an
equivariant holomorphic motion $\ol\mu:X(f)\times\Wc_f\to\C$. The
set $X(f)$ is a forward invariant subset of $J(f)$. It contains no
neutral periodic points different from $0$. Every point of $X(f)$
has at least two preimages in $X(f)$ counting multiplicities.
\end{thmA}

In Section~\ref{s:poly-like} we study polynomial-like maps. The main
result of Section~\ref{s:poly-like} is Theorem B.

\begin{thmB}
  Let $P:\C\to\C$ be a polynomial, and $Y\subset\C$ be a non-se\-pa\-ra\-ting
  $P$-invariant continuum.
  The following assertions are equivalent:
\begin{enumerate}
 \item the set $Y$ is the filled Julia set of some polynomial-like
     map $P:U^*\to V^*$ of degree $k$,
 \item $Y$ is a component of the set $P^{-1}(P(Y))$ and, for every
     attracting or parabolic point $y$ of $P$ in $Y$, the
     attracting basin of $y$ or the union of all parabolic domains
     at $y$ is a subset of $Y$.
\end{enumerate}
\end{thmB}

The proof of Theorem B uses some ideas communicated by M. Lyubich to
the fourth named author.

Finally, in Section~\ref{s:exteclosu} we prove Theorem C, in which we
show that potentially renormalizable polynomials are immediately
renormalizable depending on transparent assumptions about their
dynamics.

\begin{thmC}
Let $|\la|\le 1$ and $f\in \Fc_\la\sm \Pc_\la$. Then one of the
following holds:
\begin{enumerate}

\item $f$ belongs to the unbounded component of $\Fc_\la\sm \Pc_\la$
and $f$ is immediately renormalizable;

\item $f$ belongs to a bounded component of
    $\Fc_\la\sm \Pc_\la$ of Siegel capture type and $f$ is not
    immediately renormalizable;

\item $f$ belongs to a bounded component of
    $\Fc_\la\sm \Pc_\la$ of queer type, and there are two
    possibilities: \emph{(a)} the entire orbit of $\om_2(f)$ is
    disjoint from $X(f)$ and $f$ is immediately renormalizable, or
    \emph{(b)} $\om_2(f)\in X(f)$ and $f$ is not immediately
    renormalizable.
\end{enumerate}

If $f$ is immediately renormalizable, the corresponding
quadratic-like map is hybrid equivalent to a quadratic polynomial
$z^2+c$ with $c\in\ol\phd_2$ or, equivalently, to the polynomial
$\lambda z + z^2$.
\end{thmC}

We do not know whether any of the options mentioned in parts (2) and
(3) realizes. The authors are indebted to the referee, whose
suggestions led to the proof of part (3) of Theorem C.

There is no loss of generality in that we consider only
perturbations of $f$ in $\Fc$: instead, we could consider small
perturbations $g$ of $f$ such that, arbitrarily close to $0$, the
map $g$ has an attracting fixed point. Theorem A generalizes some
results from \cite{BH,Z}.

\medskip

{\footnotesize \emph{Notation and Preliminaries:} we write $\ol A$
for the closure of a subset $A$ of a topological space and $\bd(A)$
for the boundary of $A$; the $n$-th iterate of a map $f$ is denoted
by $f^{\circ n}$. Let $\C$ stand for the complex plane, $\C^*$ for
the Riemann sphere, $\disk$ for the open unit disk consisting of all
complex numbers $z$ with $|z|<1$, and $\uc=\bd(\disk)$ for the unit
circle, which is identified with $\R/\Z$.
The $d$-tupling map of the unit circle is denoted by $\si_d$. 
We assume knowledge of basic notions from complex dynamics, such as
\emph{Green function, dynamic rays $($of specific argument$)$,
B\"ottcher coordinate, Fatou domain, repelling, attracting, neutral
periodic points, parabolic, Siegel, Cremer periodic points} etc
(see, e.g., \cite{Mc}). }

\section{Potentially renormalizable polynomials}
\label{s:notinphd}

Throughout Section~\ref{s:notinphd}, we consider a potentially
renormalizable cubic polynomial $f$. We want to see when $f$ is
immediately renormalizable. Below we outline our strategy; to
motivate our approach, assume for the moment that $f$ \emph{is
already} immediately renormalizable. The main idea is to observe
that some points and sets related to the quadratic-like restriction
$f^*$ (including $J^*=J(f^*)$) can actually be defined independently
of the fact that $f$ is immediately renormalizable. This allows us
to define them for \emph{all} potentially renormalizable maps $f$;
in other words, we define a \emph{potential quadratic-like Julia
set} and then prove that in some cases the potential of being a
quadratic-like Julia set is realized.

Recall that $\Ac$ is the set of all cubic polynomials $g\in\Fc$ with
$|g'(0)|<1$, and, for $g\in \Ac$, we write $A(g)$ for the immediate
basin of attraction of $0$ with respect to $g$. If $g\in \Ac$ is
such that there is a unique critical point in $A(g)$, we let
$\om_1(g)$ be this critical point and let $\om_2(g)$ be the other
critical point of $g$. As a tool, we consistently approximate
potentially renormalizable maps $f$ by polynomials from $\Ac$. If a
potentially renormalizable map $f$ belongs to $\Ac$, then $f$ itself
serves as its own approximation. Since $f$ is potentially
renormalizable, there is a neighborhood of $f$ in $\Fc$, in which
there is no polynomial $g\in\Ac$ with $[g]\in\phd_3$.

First we define the critical points $\om_1(f)$ and $\om_2(f)$ for
all potentially renormalizable maps, and show that these points
depend holomorphically on $f\in\Fc_\la\sm\Pc_\la$. We next consider
a countable set $Z(f)$ of iterated $f$-preimages of $\om_1(f)$, each
of which depends holomorphically on $f\in\Fc_\la\sm\Pc_\la$. For
$g\in\Ac$ with a unique critical point $\om_1(g)\in A(g)$, the set
$Z(g)$ is the set of all iterated $g$-preimages of $\om_1(g)$
contained in $A(g)$. Finally, the potential quadratic-like Julia set
of $f$ can be defined as the set of all non-isolated points in $\ol
Z(f)$. We will show that this set moves holomorphically with $f$.

\subsection{The principal critical point of $f$}
\label{ss:princ}

Fix $f\in\Fc_\la\sm\Pc_\la$ as above.

\begin{lem}
  \label{l:2crpts}
The polynomial $f$ has two distinct critical points.
\end{lem}

\begin{proof}
Assume that $\omega(f)$ is the only critical point of $f$ (then it
has multiplicity two). Let $\Cc$ be the space of all polynomials
$g\in \Fc$ with a multiple critical point $\omega(g)$. This is an
algebraic curve in $\Fc$ passing through $f$ and given by the
formula $g_a(z)=\frac{a^2}3z+az^2+z^3, a\in \C$. This implies that
there are polynomials $g\in\Cc$ arbitrarily close to $f$, for which
$|g'(0)|<1$. The class of any such polynomial $g$ belongs to
$\phd_3$ as the immediate basin of $0$ with respect to $g$ must
contain the multiple critical point $\omega(g)$, contradicting our
assumption on $f$.
\end{proof}

By Lemma~\ref{l:2crpts}, there are two critical points of $f$.
Recall, that a critical point $c$ of $f$ is said to be
\emph{principal} if there is a neighborhood $\Uc$ of $f$ in $\Fc$
and a holomorphic function $\omega_1:\Uc\to\C$ defined on this
neighborhood such that $c=\omega_1(f)$, and, for every
$g\in\Uc\cap\Ac$, the point $\omega_1(g)$ is the critical point of
$g$ contained in $A(g)$.

\begin{thm}
  \label{t:princ}
There exists a unique principal critical point of $f$.
\end{thm}

\begin{proof}
By Lemma \ref{l:2crpts}, the two critical points of $f$ are
different. Then there are two holomorphic functions, $\omega_1$ and
$\omega_2$, defined on a convex neighborhood $\Uc$ of $f$ in $\Fc$,
such that $\omega_1(g)$ and $\omega_2(g)$ are the critical points of
$g$ for all $g\in\Uc$. Suppose that neither $\omega_1(f)$, nor
$\omega_2(f)$ is principal. Then, arbitrarily close to $f$, there
are cubic polynomials $g_1\in \Ac\cap \Uc$ and $g_2\in\Ac\cap \Uc$
with $\omega_2(g_1)\not\in A(g_1)$ and $\omega_1(g_2)\not\in
A(g_2)$. Since $A(g_i)$ contains a critical point for $i=1,2$, we
must have that $\omega_i(g_i)\in A(g_i)$.

The set $\Ac$ is convex. Therefore, the intersection $\Uc\cap\Ac$ is
also convex, hence connected. Let $\Oc_i$, $i=1$, $2$, be the subset
of $\Uc\cap\Ac$ consisting of all polynomials $g$ with
$\omega_i(g)\in A(g)$. By the preceding paragraph, $g_1\in \Oc_1$
and $g_2\in \Oc_2$. We claim that $\Oc_i$ is open. Indeed, if
$g\in\Oc_i$, then there exists a Jordan disk $U\subset A(g)$ with
$g(U)$ compactly contained in $U$, and $\omega_i(g)\in U$. If
$\tilde g\in\Uc\cap\Ac$ is sufficiently close to $g$, then $\tilde
g(U)$ is still compactly contained in $U$, and $\omega_i(\tilde g)$
is still in $U$, by continuity. It follows that $U\subset A(\tilde
g)$, in particular, $\omega_i(\tilde g)\in A(\tilde g)$. Thus,
$\Oc_i$ is open. Since $\Oc_1, \Oc_2$ are open and non-empty, the
set $\Uc\cap\Ac$ is connected, and
$$
\Uc\cap\Ac=\Oc_1\cup\Oc_2,
$$
the intersection $\Oc_1\cap\Oc_2$ is nonempty. Note that
$\Oc_1\cap\Oc_2$ consists of polynomials, whose classes are in
$\phd_3$. Since $\Uc$ can be chosen arbitrarily small, it follows
that $f$ can be approximated by maps $g\in\Ac$ with $[g]\in\phd_3$, a contradiction.

The existence of a principal critical point of $f$ is thus proved.
The uniqueness follows immediately from the fact that $f\notin
\ol{\phd_3}$.
\end{proof}

Denote by $\omega_1(f)$ the principal critical point of $f$ and by
$\om_2(f)$ the other critical point of $f$. For $g\in \Fc_{nr}$
sufficiently close to $f$, the point $\omega_1(g)$ is a holomorphic
function of $g$.

\begin{lem}\label{l:om12}
Let $m$ and $n$ be two non-negative integers. Then for each slice $\la,
|\la|\le 1$ the function $\varphi_{m,
n}(f)=f^m(\om_1(f))-f^n(\om_2(f))$ is not a constant on the slice
$\Fc_\la$.
\end{lem}

\begin{proof}
Suppose otherwise and let $\varphi_{m, n}(f)=s$. Observe that there
exists $M$ such that for each $g\in \Fc_{nr}$ with $|g'(0)|<1$ the
set $A(g)$ is contained in the disk $D$ of radius $M$ centered at
$0$. Since arbitrarily close to any $f\in \Fc_{nr}$ there are maps
$g\in \Fc_{nr}$ with $|g'(0)|<1$, it follows that the $f$-orbit of
$\om_1(f)$ is contained in $D$ for any $f\in \Fc_\la$. Since by the
assumption $\varphi_{m, n}(f)=s$ it follows that the $f$-orbit of
$\om_2(f)$ is contained in the disk $\widehat D$ of radius $M+|s|$
centered at $0$. Thus, in this case on the entire $\Fc_\la$ both
critical points are non-escaping, a contradiction.
\end{proof}

\subsection{Holomorphic motion}
\label{ss:holmot} Let $\Lambda$ be a Riemann surface, and
$Z\subset\C^*$ any (!) subset. A {\em holomorphic motion} of the set
$Z$ is a map $\mu:Z\times \Lambda\to\C^*$ with the following
properties:
\begin{itemize}
  \item for every $z\in Z$, the map $\mu(z,\cdot):\{z\}\times \Lambda\to\C^*$ is
  holomorphic;
  \item for $z\ne z'$ and every $\nu\in\Lambda$, we have $\mu(z,\nu)\ne \mu(z',\nu)$;
  \item there is a point $\nu_0$ such that $\mu(z,\nu_0)=z$ for all $z\in Z$.
\end{itemize}
We will use the following crucial \emph{$\lambda$-lemma} of Ma\~n\'e, Sad and
Sullivan \cite{MSS}: {\em a holomorphic motion of a set $Z$ extends
to a unique holomorphic motion of the closure $\ol Z$; moreover,
this extension is a continuous function in two variables such that,
for every $\nu\in \Lambda$, the map $\vp: \ol{Z}\to\C^*$ defined
as $\vp(z)=\mu(z, \nu)$ is quasi-symmetric}. There have
been useful generalizations of this result, 
but we will only need the original version. It is worth mentioning
here that the simplest version of the $\lambda$-lemma (extension to
the closure) appeared also in \cite{lyu83}.

We will now define a \emph{countable set $Z(f)$ of iterated
preimages of the principal critical point $\omega_1(f)$}. By
definition, a point $z\in\C$ belongs to $Z(f)$ if there exists an
open convex neighborhood $\Uc_z$ of $f$ in $\Fc$ and a holomorphic
function $\zeta:\Uc_z\to\C$ with the following properties:
\begin{itemize}
\item $\zeta(f)=z$;
\item we have $g^{\circ n}(\zeta(g))=\omega_1(g)$ for all $g\in\Uc_z$
and for some $n\ge 0$ independent of $g$;
\item we have $\zeta(g)\in A(g)$ for all $g\in\Uc_z\cap\Ac$.
\end{itemize}
A holomorphic function $\zeta:\Uc_z\to\C$ like above is called a
\emph{deformation} of $z\in Z(f)$. As it is always clear what kind
of deformation we consider, in what follows we will suppress the
subscript in the notation for $\Uc$.

For any $f\in \Fc_\la\sm\Pc_\la$, let $\Wc_f$ be a potentially renormalizable
component containing $f$; clearly, $\Wc_f$ is open.

\begin{lem}
  \label{l:no-crval}
Let $f$ be as above.
\begin{enumerate}
 \item The critical point $\om_1(f)$ is not eventually mapped to $\om_2(f)$.
 \item The set $Z(f)$ contains no critical values of $f$.
\end{enumerate}
\end{lem}

\begin{proof}
Suppose first that $\om_1(f)$ is eventually mapped to $\om_2(f)$, say,
$f^{\circ m}(\om_1(f))=\om_2(f)$, and the number $m$ is the minimal
positive integer with this property. Consider the set $\Cc$ of all
$g\in\Uc$ such that $g^{\circ m}(\om_1(g))=\om_2(g)$. This set is an
open part (not necessarily connected) of an algebraic curve. The
function $g\mapsto g'(0)$ is a complex analytic function on $\Cc$. By
Lemma~\ref{l:om12} this function is not a constant. Since the value of
this function at $f$ lies in $\cdisk$, there are maps $g\in\Cc$
arbitrarily close to $f$ such that $|g'(0)|<1$. The class of any such
$g$ must belong to $\phd_3$. Indeed, the attracting basin $A(g)$ must
contain the principal critical point $\om_1(g)$, by definition of the
principal critical point. Since $\om_1(g)$ is eventually mapped to
$\om_2(g)$, the critical point $\om_2(g)$ is also contained in $A(g)$.
We arrive at a contradiction with our assumption on $f$.

Suppose now that $v\in Z(f)$ is a critical value. Let
$\zeta:\Uc\to\C$ be a deformation of $v$. Consider the set $\Cc$ of
all $g\in\Uc$ such that $\zeta(g)$ is a critical value.
This set is
a part of an algebraic curve. Take a sequence $g_n\in\Cc\cap\Ac$
that converges to $f$. Since $\zeta(g_n)\in A(g_n)$ is a
critical value with at least two $g_n$-preimages in $A(g_n)$,
counting multiplicities, the set $A(g_n)$ must contain a critical
point $d_n$ with $g_n(d_n)=\zeta(g_n)$. The fact that
$\omega_1(g_n)$ is not periodic implies that $d_n\ne \omega_1(g_n)$.
Thus, both critical points of $g_n$ are contained in $A(g_n)$, and
so $[g_n]\in\phd_3$. We again arrive at a
contradiction with our assumption on $f$.
\end{proof}

\begin{lem}
  \label{l:2pts}
For every $z\in Z(f)$, there are exactly two points of $Z(f)$ that
are mapped to $z$ under $f$.
\end{lem}

\begin{proof}
The proof is similar to that of Theorem \ref{t:princ}. Let
$\zeta:\Uc\to\C$ be a deformation of $z$. Since the set $Z(f)$
cannot contain a critical value of $f$, there are three holomorphic
functions $\zeta_1$, $\zeta_2$, $\zeta_3$ defined on $\Uc$ and such
that $g(\zeta_i(g))=\zeta(g)$ (we may need to pass to a smaller
neighborhood $\Uc$ to arrange this).

The intersection $\Uc\cap\Ac$ is convex, hence connected. For any
2-element subset $\{i,j\}\subset\{1,2,3\}$, define a subset
$\Oc_{ij}\subset\Uc\cap\Ac$ as the set of all polynomials
$g\in\Uc\cap\Ac$ such that $\zeta_i(g)\in A(g)$ and $\zeta_j(g)\in
A(g)$. All three sets $\Oc_{12}$, $\Oc_{23}$ and $\Oc_{13}$ are open
(cf. the proof of Theorem \ref{t:princ}). On the other hand, we have
$$
\Ac\cap\Uc=\Oc_{12}\cup\Oc_{23}\cup\Oc_{13}.
$$
Hence either only one of the sets $\Oc_{ij}$ is nonempty, or at
least two of the sets $\Oc_{ij}$ intersect. In the latter case,
$\zeta_i(g)\in A(g)$ for some $g\in\Ac\cap\Uc$ and all $i=1$, $2$,
$3$. It follows that $[g]\in\phd_3$. Since the neighborhood $\Uc$
can be chosen to be arbitrarily small, it follows that
$f$ can be approximated by polynomials in $\Ac$, whose classes are in $\phd_3$,
a contradiction. The contradiction shows that
only one of the sets $\Oc_{ij}$ is nonempty, for a suitable choice
of the neighborhood $\Uc$. Assume that $i=1$ and $j=2$; then
$\zeta_1(f)$, $\zeta_2(f)\in Z(f)$ but $\zeta_3(f)\not\in Z(f)$.
\end{proof}

The proof of Lemma \ref{l:2pts} implies a stronger claim below.

\begin{cor}
  \label{c:2pts}
Let $\zeta_i$ be holomorphic functions introduced in the proof of
Lemma \ref{l:2pts}. Suppose that $\zeta_1(f)$, $\zeta_2(f)\in Z(f)$.
Then there is a neighborhood $\Uc$ of $f$ in $\Fc$ such that
$\zeta_3(g)\not\in A(g)$ for all $g\in\Uc\cap\Ac$.
\end{cor}

We will need the following lemma.

\begin{lem}\label{l:tophd}
If $|\la|<1$, then the set of all polynomials $f\in \Fc_\la$
such that both critical points of $f$ belong to the immediate basin of
attraction of $0$ is an open simply connected domain. Moreover, if
$|\la|\le 1$, then the set $\Pc_\la$ is connected.
\end{lem}

\begin{proof}
Suppose that $|\la|<1$.
Consider the branched covering map $\Psi$ of $\C$ by $\Fc_\la$ sending $f_{\la,b}$ to $b^2$.
The fibers of this map are exactly affine conjugacy classes in $\Fc_\la$.
By parts b2) and c) of \cite[Theorem $A'$]{peta09}, the $\Psi$-image of $\Pc_\la$ is a closed Jordan disk
(see also \cite{pita04}).
The set $\Pc_\la$ is therefore a disk or a pair of disks depending on whether or not the polynomial $f=f_{\la,0}(z)=\la z+z^3$ belongs to $\Pc_\la$.
Observe that $f$ is odd, therefore, critical points are antipodal and the basin $A(f)$ of immediate attraction
of $0$ is centrally symmetric with respect to the origin. Since at least one of the critical points must belong
to $A(f)$, then both do which implies that $f\in \Pc_\la$ as desired.

Suppose now that $|\la|\le 1$.
By \cite{peta09}, the set $\Pc_{\la'}$ depends continuously on $\la'$.
The set $\Pc_\la$ is the upper limit of $\Pc_{\la'}$ as $\la'\to\la$.
Since the upper limit of a continuous family of connected sets (over a locally connected base) is connected, the set $\Pc_\la$ is connected.
\end{proof}

Lemma~\ref{l:tophd} was included at the suggestion of the referee; we are indebted to him/her
for this suggestion.

\begin{prop}
  \label{p:Lambda}
For every $z\in Z(f)$, there is a holomorphic function
$\zeta:\Wc_f\to\C$ such that $\zeta(h)\in Z(h)$ for all
$h\in\Wc_f$ and $\zeta(f)=z$.
\end{prop}

\begin{proof}
The function $\zeta$ with these properties is defined at least on
some open neighborhood of $f$ in $\Wc_f$, by definition of the set
$Z(f)$. Assume by induction that the statement of the proposition
holds for the point $f(z)$, i.e., there is a holomorphic function
$\eta:\Wc_f\to\C$ such that $\eta(h)\in Z(h)$ for all $h\in\Wc_f$
and $\eta(f)=f(z)$. It follows that there is an integer $n$ such
that $h^{\circ (n-1)}(\eta(h))=\omega_1(h)$ for all $h\in\Wc_f$.
Consider the multivalued analytic function $h\mapsto
h^{-1}(\eta(h))$. If this function has no branch points in $\Wc_f$,
then we can define the holomorphic function $\zeta$ as the branch of
this function such that $\zeta(f)=z$. Suppose that there is a branch
point $h_0$ of the multivalued function $h\mapsto h^{-1}(\eta(h))$.
Then the point $\eta(h_0)$ is a critical value of $h_0$, a
contradiction with Lemma \ref{l:no-crval}.

By Lemma \ref{l:tophd}, the set $\Wc_f$ is simply connected, therefore, a multivalued analytic function without branching is necessarily a disjoint union of single valued branches.
Thus, there is a holomorphic function $\zeta:\Wc_f\to\C$ with
$h(\zeta(h))=\eta(h)$, and $\zeta(f)=z$. Moreover, $\zeta(h)\in Z(h)$
for all $h\in\Wc_f$ sufficiently close to $f$. It suffices to prove
that $\zeta(h)\in Z(h)$ for all $h\in\Wc_f$. To this end, we will prove
that the set of polynomials $h\in\Wc_f$ such that $\zeta(h)\in Z(h)$ is
open and closed in $\Wc_f$. The openness is obvious. Consider a
sequence $h_n\in\Wc_f$ converging to some polynomial $h\in\Wc_f$, and
suppose that $\zeta(h_n)\in Z(h_n)$ but $\zeta(h)\not\in Z(h)$.
Therefore, there are two other holomorphic functions $\zeta_1$,
$\zeta_2$ defined on some neighborhood of $h$ such that $\zeta_i(h)\in
Z(h)$, $i=1$, $2$. It follows that $\zeta_i(h_n)\in Z(h_n)$ for
sufficiently large $n$. But then all three points $\zeta_1(h_n)$,
$\zeta_2(h_n)$ and $\zeta(h_n)$ are preimages of $\eta(h_n)$ in
$Z(h_n)$. This contradicts Lemma \ref{l:2pts}.
\end{proof}

Proposition \ref{p:Lambda} and Lemma~\ref{l:no-crval} imply the
following theorem.

\begin{thm}
  \label{t:holmot}
  There exists a holomorphic motion $\mu:Z(f)\times\Wc_f\to\C$
  that is \emph{equivariant} in the sense that for every $h\in\Wc_f$,
  and for every $z\in Z(f)\setminus\{\omega_1(f)\}$,
  we have $h(\mu(z,h))=\mu(f(z),h)$.
\end{thm}

By the $\lambda$-lemma, the holomorphic motion $\mu$ gives rise to
the holomorphic motion $\ol\mu:\ol{Z(f)}\times\Wc_f\to\C$. Since
$\mu$ is equivariant, the holomorphic motion $\ol\mu$ is equivariant
too.

\subsection{The set $X(f)$}
\label{ss:X_f} Let $Z_n(f)$ be the subset of $Z(f)$ consisting of
all preimages of $\omega_1(f)$ mapped to $\omega_1(f)$ in $n$ steps,
in other words, $z\in Z_n(f)$ if $f^{\circ n}(z)=\omega_1(f)$. Define the set
$X(f)$ as the limit of the sets $Z_n(f)$, i.e.,
$$
X(f)=\bigcap_{m\ge 0}\ol{\bigcup_{n\ge m}Z_n(f)}.
$$
It is easy to see that $X(f)$ coincides with the set of all
non-isolated points in $\ol{Z(f)}$.
Theorem \ref{t:holmot} and the $\lambda$-lemma imply that the sets
$X(h)$ move holomorphically for $h\in\Wc_f$. Clearly, $X(h)$ is
forward invariant under $h$.

Let $P$ be a polynomial of degree $d$. Then $P$ on a small
neighborhood of any point $t$ is $k$-to-$1$ (at regular points $t$,
we have $k=1$, and at critical points $k>1$); the number $k$ is called the
\emph{multiplicity of $t$}.

\begin{lem}
  \label{l:Xf-2J}
Let $h\in \Wc_f$. Then the set $X(h)$ is a subset of the Julia set
$J(h)$, and every point $x\in X(h)$ has at least two preimages in
$X(h)$, counting multiplicities.
\end{lem}

\begin{proof}
The set $X(h)$ is contained in the accumulation set of the backward
orbit of $\omega_1(h)$. The backward orbit of a point can accumulate in
the Fatou set only if the point lies in a Siegel disk. However
$\omega_1(h)$ cannot lie in a Siegel disk as a Siegel disk contains no
critical points. The second claim follows from Lemma \ref{l:2pts}.
\end{proof}

Recall that, by the $\la$-lemma, $\ol\mu:\ol {Z(f)}\times \Wc_f\to
\C$ is continuous. In particular, if a sequence $z_n\in \ol{Z(f)}$
converges to $z\in \ol{Z(f)}$, then $\mu(z_n,h)$ converges to
$\mu(z,h)$, for every $h\in\Lambda$.

\begin{lem}
\label{l:no-par}
  The set $X(f)$ contains no neutral periodic points different from $0$.
\end{lem}

\begin{proof}
Let $X(f)$ contain a periodic neutral point $x\ne 0$ of minimal
period $k$. Since the holomorphic motion $\ol\mu$ is equivariant,
$\ol\mu(x,h)=\ol\mu(f^{\circ r}(x), h)=h^{\circ r}(\ol\mu(x, h))$
for every $r$. This proves that $\ol\mu(x, h)=x(h)$ is a periodic
point of $h$ of period $k$, for every $h\in\Wc_f$.

The holomorphic function $h\mapsto (h^{\circ k})'(x(h))$ is
non-constant on the multiplier slice $\Fc_\lambda$. Indeed, the
slice $\Fc_\lambda$ contains polynomials with disconnected Julia
sets, and such polynomials cannot have non-repelling periodic points
different from $0$, by the Fatou--Shishikura inequality. It follows
that $x(h)$ is an attracting periodic point with respect to $h$, for
some polynomials $h$ in arbitrarily small neighborhood of $f$, a
contradiction to Lemma~\ref{l:Xf-2J}.
\end{proof}

Theorem A explicitly summarizes the results of this section.

\begin{thmA}
Suppose that $f\in \Wc_f$ where $\Wc_f$ is a potentially
renormalizable component of $\Fc_\la\sm \Pc_\la$. Then there is an
equivariant holomorphic motion $\ol\mu:X(f)\times\Wc_f\to\C$. The
set $X(f)$ is a forward invariant subset of $J(f)$. It contains no
neutral periodic points different from $0$. Every point of $X(f)$
has at least two preimages in $X(f)$ counting multiplicities.
\end{thmA}

Finally in this section we state Lemma~\ref{l:motiv}.

\begin{lem}\label{l:motiv}
Suppose that $f\in \Fc_{nr}$ is immediately renormalizable and $J^*(f)$
is its quadratic-like Julia set. Then $J^*(f)=X(f)$ and the principal
critical point $\om_1(f)$ belongs to $\thu(X(f))$.
\end{lem}

\begin{proof} Left to the reader.
\end{proof}

In the rest of the paper, we adopt the following approach. First we
establish several types of conditions on $X(f)$ and the holomorphic
motion $\ol\mu$ sufficient for $f$ being immediately renormalizable;
the set $X(f)$ plays here the role of a potential quadratic-like
Julia set. Then we verify that these conditions are fulfilled for
various cubic polynomials. In the end, this will lead to the proofs
of our results.

\section{Properties of polynomial-like maps}\label{s:poly-like}

In this section, we prove a criterion for a polynomial $P$ of any
degree to have a polynomial-like restriction.
Recall that, for any map $F$, by an
\emph{$F$-invariant} set, we mean a set $A$ such that $F(A)\subset
A$ but not necessarily $F(A)=A$.

We start with some purely topological considerations.
Consider a compactum $T\subset\C^*$ and a branch covering $P:\C^*\to\C^*$.
Let $\nu_T(z)$ be the number of all $P$-preimages of $z$ in $T$ counted
with multiplicities. Then there exists a neighborhood $V$ of $z$ and
$r$ pullbacks $W_1$, $\dots$, $W_r$ of $V$ each containing exactly one
point of the set $P^{-1}(z)\cap T$ and such that the sum of degrees of
$P$ restricted on $W_1$, $\dots$, $W_r$ is $\nu_T(z)$. If a point $x\in
P^{-1}(z)\cap T$ is not critical and belongs to $W_i$, the map
$P|_{W_i}$ is a homeomorphism onto image. If a point $x\in
P^{-1}(z)\cap T\cap W_i$ is critical, then it is the unique critical
point of $P$ in $W_i$.

Set $\widehat W=\cup^r_{i=1}W_i$. By compactness, the $P$-image of
$T\sm \widehat W$ is positively distant from $z$. Hence, for some
smaller neighborhood $V'\subset V$ of $z$, \emph{all} preimages of
any point $z'\in V'$ in $T\cap P^{-1}(V')$ belong to $\widehat W$
and the entire preimage of $V'$ in $T$ breaks down into $r$ pieces
contained in $W_1, \dots, W_r$. From now on let us call such $V'$ a
\emph{$(T-)$suitable} neighborhood of $z'$. Since \emph{any} point
$y\in V$ has exactly $\nu_T(z)$ preimages in $\widehat W$ (not
necessarily in $T$), the value of $\nu_T$ can only drop at points
$z'\in V'$, and $\nu_T$ is upper-semicontinuous.

Let $\nu_T|_{P(T)}$ be continuous at $z$. Choose a suitable
neighborhood $V$ of $z$, on which $\nu_T|_{P(T)}$ is constant.
Then, for every point $y\in V\cap P(T)$, the set $P^{-1}(y)\cap T$
of its preimages in $T$ consists exactly of \emph{all} its preimages
in $\widehat W$. Indeed, by the previous paragraph, if $y\in V$, then
it has $\nu_T(z)$ preimages in $\widehat W$. Together with the fact
that $\nu_T(y)=\nu_T(z)$ this implies our claim. Hence if $x\in
P(T)$ is a point of continuity of $\nu_T|_{P(T)}$, then $x$ has a
neighborhood $V$ such that in the corresponding open set $\widehat
W$ points from $T$ and not from $T$ \emph{cannot} have the same
image. If $\nu_T|_{P(T)}$ is continuous at all points, it follows
that there exists a neighborhood $U$ of $T$ such that for any $z\in
U\sm T$ we have $P(z)\notin P(T)$. It is equally easy to show that,
conversely, if $\nu_T|_{P(T)}$ is discontinuous at $z$, then there
is a preimage of $z$ in $T$, in whose arbitrarily small neighborhood
``collisions'' between a point from $T$ and a point not from $T$
take place. All this is summarized in Lemma~\ref{l:om2notin}.

\begin{lem}\label{l:om2notin}
Suppose that $T\subset\C^*$ is a compact set and $P:\C^*\to\C^*$ is a branch covering.
Then the following two properties are equivalent.
\begin{enumerate}
\item The function $\nu_T$ is continuous on $P(T)$.
\item There exists a neighborhood $U$ of $T$ such that for any
$z\in U\sm T$ we have $P(z)\notin P(T)$.
\end{enumerate}

If $T$ is connected, these conditions are equivalent to the
following:

\begin{enumerate}


\item [$(3)$] $T$ is a component of $P^{-1}(P(T))$.

\item [$(4)$] $\nu_T$ is a constant on $T$.

\end{enumerate}

\end{lem}

\begin{proof}
By the arguments right before Lemma~\ref{l:om2notin}, claims (1) and (2) are
equivalent. Assume now that $T$ is connected. Then (1) and (4) are
equivalent, and (2) implies (3). Suppose that (3) holds. Then there is
a neighborhood $U$ of $T$ that does not intersect other components of
$P^{-1}(P(T))$. It follows for every $z\in U\sm T$ we have $P(z)\notin
P(T)$.
\end{proof}

There is a useful sufficient condition for (3) in the polynomial case.

\begin{lem}\label{l:suffnotin}
Suppose that $T$ is a continuum and $P$ is a polynomial. Set
$m=1+\sum (d_c-1)$, where the sum is taken over all critical points
$c$ of $P$ in $T$, and $d_c$ is the multiplicity of the point $c$.
If $(1)$ $P(c)\notin \thu(P(T))$ for any critical point $c\notin
\thu(T)$, $(2)$ there are no critical points in $\thu(T)\sm T$, and
$(3)$ for each point $x\in P(T)$ we have $\nu_T(x)\ge m$, then $T$
is a component of the set $P^{-1}(P(T))$ and $\thu(T)$ is a
component of the set $P^{-1}(P(\thu(T)))$ (in particular,
$\thu(P(T))=P(\thu(T))$).
\end{lem}

\begin{proof}
Take all critical points of $P$ not belonging to $\thu(T)$, connect
their $P$-images to infinity with pairwise disjoint simple curves
(``rays'') avoiding $\thu(P(T))$, and pull them back to their
critical points to construct a finite collection of cuts of the
plane. Let $W$ be a complementary component of this collections of
cuts containing $T$. Clearly, $P:W\to P(W)$ is a branched covering
map. By the Riemann-Hurwitz formula, the topological degree of this
map is $m$ (observe that by the construction and by the assumptions
and there are no critical points in $W\sm T$). Thus points of $P(T)$
can have at most $m$ preimages in $T$. By the assumptions, this
implies that they have exactly $m$ preimages in $T$ counting
multiplicities.

Let us show that $T$ is a component of the set $P^{-1}(P(T))$.
Indeed, by the previous paragraph condition (4) from
Lemma~\ref{l:om2notin} is fulfilled; hence, Lemma~\ref{l:om2notin}
implies the desired. Moreover, $T$ is the unique component of the
set $P^{-1}(P(T))$ in $W$ (recall that the topological degree of
$P|_W$ is $m$). Let us show that $P(\thu(T))\subset \thu(P(T))$.
Indeed, suppose otherwise. Then there exists a component $V$ of
$\thu(T)\sm T$ and a point $x\in V$ such that $P(x)\notin
P(\thu(T))$. Connect $P(x)$ to infinity with a simple curve
(``ray'') avoiding $\thu(P(T))$ and pull it back to $x$. This gives
rise to a ``ray'' connecting $x$ to infinity and avoiding $T$, a
contradiction. Thus, $P(\thu(T))\subset \thu(P(T))$.

Now, suppose that a point $y\in \thu(P(T))$ has a preimage $z\in
W\sm \thu(T)$. By the above, this can only happen if $y\in
\thu(P(T))\sm P(T)$. Let $U$ be the component of $\thu(P(T))\sm
P(T)$ such that $y\in U$. Consider a component $Q$ of $P^{-1}(U)$
such that $z\in Q$. Since $z\notin \thu(T)$, it follows that there
are points of $\bd(Q)\subset W$ which do not belong to $T$ while
their images belong to $\bd(U)\subset T$. This contradicts the fact
that $T$ is the unique component of the set $P^{-1}(P(T))$ in $W$
and completes the proof of the lemma).
\end{proof}

Theorem B is the main result of this section. Recall that a
\emph{parabolic domain} at a periodic parabolic point $y$ is a periodic
Fatou component that contains $y$ in its boundary and whose points converge to $y$ under the iterates of the given polynomial.

\begin{thmB}
  Let $P:\C\to\C$ be a polynomial, and $Y\subset\C$ be a full
  $P$-invariant continuum. The following assertions are equivalent:
\begin{enumerate}
 \item the set $Y$ is the filled Julia set of some polynomial-like
     map $P:U^*\to V^*$ of degree $k$,
 \item $Y$ is a component of the set $P^{-1}(P(Y))$, and, for every
     attracting or parabolic point $y$ of $P$ in $Y$, the
     immediate attracting basin of $y$ or the union of all parabolic domains
     at $y$ is a subset of $Y$.
\end{enumerate}
\end{thmB}

The proof uses some ideas communicated by M. Lyubich to the fourth
named author.

\begin{proof}
It suffices to prove $(2)\Longrightarrow (1)$. Let
$\phi:\disk\to\C^*\sm Y$ be a Riemann map. By
Lemma~\ref{l:om2notin}, a point $x\not\in Y$ close to $Y$ cannot map
into $Y$. Hence we can choose $\eps>0$ so that the map
$F=\phi^{-1}\circ P\circ \phi$ is defined and holomorphic on the
annulus $A_\eps=\{z: 1-\eps<|z|<1\}$. Moreover, the map $F$ extends
continuously to the unit circle $\{|z|=1\}$. Indeed, the map $\phi$
induces a homeomorphism $\widehat\phi$ between the set of prime ends
of $\C^*\sm Y$ and the unit circle. Note that $P$ induces a
continuous map $\widehat P$ on the prime ends of $\C^*\setminus Y$.
The continuous extension of $F$ is obtained by conjugating the map
$\widehat P$ by the homeomorphism $\widehat\phi$.

By the Schwarz reflection principle, we can extend the map $F$ to a
holomorphic map of the annulus $1-\eps<|z|<(1-\eps)^{-1}$ to $\C$
preserving $\uc$ (hence taking this annulus to another annulus
around $\uc$). By a theorem of Ma\~n\'e \cite{Ma}, if $F$ has no
attracting or parabolic periodic points on $\uc$, and no critical
points on $\uc$, then $F$ is expanding, i.e., $|(F^{\circ
n})'(z)|\ge C\mu^n$ for some $C>0$ and $\mu>1$.

Since $F$ takes $A_\eps$ to a subset of the disk $|z|<1$, it has no
critical points on $\uc$. Suppose that $F$ has an attracting or a
parabolic periodic point $z$ of period $r$ on $\uc$. In both cases,
there is a convex Jordan domain $\widetilde E$ such that $F^{\circ
r}(\widetilde E)\subset\widetilde E$, the closure of $\widetilde E$
contains $z$, and all points of $\widetilde E$ converge to $z$ under
the iterations of $F^{\circ r}$. Since the unit circle is invariant
under $F$, by the local theory of parabolic points, we can arrange
that $\widetilde E\cap\disk\ne\0$. Note that $\widetilde E$ and the
unit disk intersect over a convex Jordan domain $E$. By definition,
$F^{\circ r}(E)\subset E$, and all points in $E$ converge to $z$
under the iterations of $F^{\circ r}$.

Set $B=\phi(E)$. Then $P^{\circ r}(B)\subset B$. By the
Denjoy--Wolff theorem, all points of $B$ converge under the
iterations of $P^{\circ r}$ to a $P^{\circ r}$-fixed point $x\in
\bd(B)$. Clearly, $x\in Y$, and $x$ is either attracting or
parabolic (as it attracts an open set of points). However, by the
assumptions, the attracting basin of $x$ or the union of all
parabolic domains at $x$ is a subset of $Y$, a contradiction.


Thus, $F$ expands on $\uc$, and $\eps$ can be chosen so that the
$F$-pullback of $A_\eps$ is compactly contained in $A_\eps$. Let
$V^*$ be the Jordan domain bounded by the $\phi$-image of the curve
$|z|=1-\eps$. Set $U^*$ to be the component of $P^{-1}(V^*)$
containing $Y$. Then $\ol{U^*}\subset V^*$, and $P:U^*\to V^*$ is a
polynomial-like map. The fact that $Y$ is the filled Julia
set of this polynomial-like map, follows easily.
\end{proof}

\section{Proof of Theorem C}\label{s:exteclosu}


First we study the unbounded potentially renormalizable component.

\begin{cor}
  \label{c:unbound}
  The unbounded potentially renormalizable component
  consists of immediately renormalizable maps.
\end{cor}

\begin{proof}
  Let $\Wc_\infty$ be the unbounded potentially renormalizable component in $\Fc_\lambda$.
  Then there is a polynomial $f\in\Wc_\infty$, whose Julia set is disconnected.
  Such a polynomial is necessarily immediately renormalizable, by \cite[Theorem 5.3]{BrHu2}.
  By Theorem A, it follows that for any $g\in \Wc_\infty$ the
  restriction $g|_{X(g)}$ is quasi-symmetrically conjugate to
  $f|_{X(f)}$ and therefore, in fact, to a quadratic polynomial from the
  Main Cardioid restricted on its Julia set. Thus, condition (1) of
  Theorem B is satisfied. Moreover, it follows that
  $\thu(X(g))$ does not contain attracting periodic points except for,
  possibly, $0$. Finally, by Theorem A, it follows that
  $X(g)$ contains no parabolic periodic points except, possibly, for
  $0$. Thus, condition (2) of Theorem B is
  satisfied too. Applying Theorem B to $g$, we see that
  $g|_{X(g)}$ is quadratic-like and $g$ is immediately renormalizable
  as desired.
\end{proof}

The following useful lemma was suggested to us by the referee which
we gratefully acknowledge here.

\begin{lem}\label{l:useful}
Let $\Vc\subset \Fc_\la$ be a stable component. Then no map from
$\Vc$ can have a parabolic periodic point other than, possibly, $0$.
\end{lem}

\begin{proof}
Suppose that $g\in \Vc$ has a parabolic periodic point $\zeta(g)\ne
0$. Then $g$ is isolated in the set of parameters in $\Fc_\la$ such
that for the associated maps the point corresponding to $\zeta(g)$
is parabolic. Thus, in a small punctured neighborhood of $g$ in
$\Fc_\la$ the multiplier of the associated map at the point
corresponding to $\zeta(g)$, has absolute value greater than $1$. By
the Maximum Principle this implies that $\zeta(g)$ cannot be
parabolic.
\end{proof}

Let us now study bounded potentially renormalizable components. We
need the notion of an active critical point introduced by McMullen
in \cite{Mc00}. Set $i=1$ or $2$, and take $f\in\Fc_\lambda$. The
critical point $\om_i(f)$ is \emph{active} if, for every
neighborhood $\Uc$ of $f$ in $\Fc_\lambda$, the sequence of the
mappings $g\mapsto g^{\circ n}(\om_i(g))$ fails to be normal in
$\Uc$. If the critical point $\om_i(f)$ is not active, then it is
said to be \emph{passive}.


\begin{cor}
\label{c:bdd-stab} Let $|\lambda|\le 1$. Every bounded potentially
renormalizable component $\Wc$ in $\Fc_\lambda$ consists of stable
maps. If $\Wc$ contains an immediately renormalizable map, then it
coincides with a stable component. The union $\Tc_\la$ of $\Pc_\la$
with all domains of stability, whose boundaries are contained in
$\Pc_\la$, equals $\thu(\Pc_\la)$.
\end{cor}

\begin{proof}
By \cite{MSS}, to prove that $f\in\Wc$ is stable, it suffices to show
that both critical points of $f$ are passive. Note that, if
$g\in\bd(\Wc)$, then the $g$-orbits of $\om_1(g)$ and of $\om_2(g)$ are
bounded uniformly with respect to $g$. By the maximum principle, the
$f$-orbits of $\om_1(f)$ and $\om_2(f)$ are uniformly bounded for all
$f\in\Wc$, which implies normality. Thus both critical points are
passive, and the first claim of the corollary is proved.

Consider now a stable component $\Vc$ containing an immediately
renormalizable polynomial $f$. We claim that then no polynomial
$g\in \Pc_\la\cap \Vc$ exists. Indeed, suppose otherwise. The fact
that $f$ is immediately renormalizable implies by
Lemma~\ref{l:motiv} that $X(f)=J^*$ is its quadratic-like Julia set.
Since $\Vc$ is stable, then $g|_{J(g)}$ and $f|_{J(f)}$ are
conjugate. Let $Y$ be the continuum corresponding to $X(f)$ under
this conjugacy. We want to apply to $Y$ Theorem B. Indeed, by
construction of $Y$ and properties of $X(f)$ it follows that
$g(Y)=Y$ and $Y$ is a component of the set $f^{-1}(Y)$.

Let us show that $g$ does not have any attracting periodic orbits
other than possibly $0$. Indeed, suppose otherwise. Then, since any
attracting periodic orbit persists under small perturbations, it
follows that arbitrarily close to $g$ there are polynomials $h$ with
$[h]\in \ol{\phd_3}$ and an attracting periodic orbit distinct from
$\{0\}$, a contradiction. Moreover, by Lemma~\ref{l:useful} $g$
cannot have a parabolic periodic point not equal to $0$.

Thus, Theorem B applies to $Y$ and implies that $Y$ is a
quadratic-like Julia set. Clearly, this contradicts the assumption
that $g\in \Pc_\la$.  Thus, if $\Vc$ is a stable component
containing an immediately renormalizable polynomial then no
polynomial $g\in \Pc_\la\cap \Vc$ exists. Since by
Corollary~\ref{c:unbound} all polynomials in $\C\sm \thu(\Pc_\la)$
are immediately renormalizable, it follows that stable components
are either contained in $\Pc_\la$, or contained in $\C\sm \Pc_\la$.
In particular, the union $\Tc_\la$ of $\Pc_\la$ with all domains of
stability, whose boundaries are contained in $\Pc_\la$, equals
$\thu(\Pc_\la)$ (recall that by the above every bounded potentially
renormalizable component $\Wc$ in $\Fc_\lambda$ consists of stable
maps). 
\end{proof}

We do \emph{not} prove that all components of $\thu(\Pc_\la)\sm
\Pc_\la$ are stable components. This we can guarantee only for the
components containing an immediately renormalizable map. Otherwise
it might happen that a component $\Wc$ of $\thu(\Pc_\la)\sm \Pc_\la$
is a proper subset of a stable component $\Vc$. However, we do prove
that every stable component is contained in $\thu(\Pc_\la)$.
Clearly, Corollary~\ref{c:bdd-stab} together with
Theorem~\ref{t:extend} implies Lemma~\ref{l:extend-th-i}.

\begin{cor}\label{c:boring}
Let $f\in \Uc$ where $\Uc$ is a bounded component of
$\thu(\Pc_\la)\sm\Pc_\la$. Then $f$ does not have non-repelling
periodic points except, possibly, for $0$. Moreover, $f$ does not
have repelling periodic cutpoints in $J(f)$.
\end{cor}

\begin{proof}
By Lemma~\ref{l:extend-th-i}, $[f]\in \cu$. Hence $f$ does not have
repelling periodic cutpoints in $J(f)$. Suppose that $f$ has a
non-repelling periodic point $y\ne 0$. Since $[f]\in \cu$, then the
only possibility for $y$ is that $y$ is parabolic. However this
contradicts the fact that $f$ belongs to a stable component (which
follows from Corollary~\ref{c:bdd-stab}) and Lemma~\ref{l:useful}.
\end{proof}

We are ready to discern in which bounded potentially renormalizable
components maps are immediately renormalizable.

\begin{dfn}\label{d:ql-capt}
Assume that $f$ is immediately renormalizable with filled
quadratic-like Julia set $K^*$. If there exists the smallest $n$ such
that $f^{\circ n}(\om_2(f))\in K^*$, then we say that $f$ is an
immediately renormalizable polynomial \emph{of capture type}. Denote
$n$ above by $n_f$; denote by $Q_f$ the pullback of $K^*$ containing
$\om_2(f)$.
\end{dfn}

Observe that $n_f>0$ by the definition of a quadratic-like map. Also,
pullbacks of $K^*$ either coincide with $K^*$ or are disjoint from
$K^*$. In particular, $Q_f$ is disjoint from $K^*$ by the definition of
a pullback. All pullbacks of $K^*$ form a family of pairwise disjoint
subcontinua of $K(f)$.

\begin{lem}\label{l:n>1}
Let $f$ be an immediately renormalizable polynomial of capture type.
Then $n_f>1$.
\end{lem}

\begin{proof}
If $n_f=1$, then $f(\om_2(f))\in K^*$ will have \emph{two} preimages in
$K^*$ and \emph{two} more preimages in $Q_f$, a contradiction (recall
that $f$ is cubic and that we count preimages with multiplicities).
\end{proof}

Lemma~\ref{l:no-ql-capt} is used in the proof of Theorem C.

\begin{lem}\label{l:no-ql-capt}
An immediately renormalizable polynomial of capture type cannot
belong to a bounded potentially renormalizable component.
\end{lem}

\begin{proof}
Let $f\in\Fc_\lambda, |\lambda|\le 1$, belong to a bounded
potentially renormalizable component $\Wc_f=\Wc$.
Assume that $f$ is of capture type and is immediately renormalizable.
The set $K^*$ and all its pullbacks contain lots of repelling
periodic points and their preimages (such points are dense in the
corresponding quadratic-like Julia set $J^*$ and its pullbacks).
Obviously, we can choose two rays $R_x$, $R_y$ with arguments $\al$ and
$\be$ landing at preperiodic points $x, y\in Q_f$ such that $f(x)=f(y)$
and $f(R_x)=f(R_y)$. The union $R_x\cup Q_f\cup R_y$ cuts the plane
into two pieces denoted by $L$ and $T$. Assume that $K^*\subset L$.
Since $f|_{K^*}$ is two-to-one, we may then assume that the arc $(\al,
\be)\subset \uc=\R/\Z$ is of length $\frac23$ and contains all angles,
whose rays are contained in $L$.

Set $F=f^{\circ (n_f-1)}$.
Clearly, $Z=F(Q_f)\subset T$ (observe that $f(Z)=K^*$
while $Z$ and $K^*$ are disjoint) because, by construction, $f(T)$ covers
$K^*$ while $f^{-1}(K^*)$ consists of two components, $K^*$ and $Z$. By
Theorem 7.5.2 of \cite{bfmot12}, this implies that $T$ contains either
a repelling $F$-fixed cutpoint of $J(f)$ or a
non-repelling $F$-fixed point.
A particular case of \cite[Theorem 7.5.2]{bfmot12} that is enough in our case
states roughly the following: if a continuum $T$ is carved out of $K(f)$ by several
``cut continua'' mapping ``towards'' $T$ under $F$, then
$T$ contains a repelling fixed cutpoint or a non-repelling fixed point for $F$.
This contradicts Corollary~\ref{c:boring} and completes the proof.
\end{proof}

We are ready to prove Theorem C.

\begin{proof}[Proof of Theorem C]
Claim (1) of Theorem C is established in Corollary~\ref{c:unbound}.
Claim (2) of Theorem C follows from Lemma~\ref{l:no-ql-capt}.

Consider now claim (3) of Theorem C. Let $f\in \Wc$ be a potentially
renormalizable polynomial of queer type. By definition, this implies
that $\om_2(f)\in J(f)$. Clearly, if $\om_2(f)\in X(f)$, then $f$ is
not immediately renormalizable (by Lemma~\ref{l:motiv}, if $f$ is
immediately renormalizable, then its quadratic-like Julia set
coincides with $X(f)$, and $\om_1(f)$ is the unique critical point
in $\thu(X(f))$). Now, assume that $\om_2(f)\notin X(f)$. Since
$\thu(X(f))\sm X(f)$ is disjoint from $J(f)$, then $\om_2(f)\notin
\thu(X(f))$. Suppose that $f(\om_2(f))\in X(f)$. Then $f(\om_2(f))$
has at least four preimages (counting with multiplicities): at least
two in $X(f)$ and $\om_2(f)$ counted twice, a contradiction. Hence
$f(\om_2(f))\notin X(f)$, and, since $\thu(X(f))\sm X(f)$ is
disjoint from $J(f)$, then $f(\om_2(f))\notin \thu(X(f))$.

It follows now from Lemma~\ref{l:suffnotin} that the set
$\thu(X(f))$ is a component of the set $f^{-1}(f(\thu(X(f)))$.
Moreover, by Corollary~\ref{c:boring}, the polynomial $f$ does not
have non-repelling periodic points except, possibly, for $0$. Hence
Theorem B applies to $\thu(X(f))$ and shows that $f$ is immediately
renormalizable. By Lemma~\ref{l:no-ql-capt}, all this implies that
the entire orbit of $\om_2(f)$ is disjoint from $\thu(X(f))$ as
desired.

On the other hand, by Lemma~\ref{l:no-ql-capt} a bounded potentially
renormalizable component of Siegel capture type cannot contain an
immediately renormalizable polynomial. Since by
Theorem~\ref{t:phd-info} a potentially renormalizable component is
either of queer type or of Siegel capture type, this completes the
proof of Theorem C.
\end{proof}

\section*{Acknowledgements}
We are grateful to M. Lyubich for useful discussions.
We would also like to thank the referee for extensive suggestions,
which have led to a significant improvement of the paper.

The first and the third named authors were partially supported by NSF
grant DMS--1201450. The second named author was partially  supported by
NSF grant DMS-0906316. The fourth named author was partially supported
by the Dynasty Foundation grant, RFBR grants 13-01-12449, 13-01-00969.
The research comprised in Theorem B was funded by RScF grant 14-21-00053.


\begin{thebibliography}{9999}





\bibitem[BFMOT12]{bfmot12} A. Blokh, R. Fokkink, J. Mayer, L. Oversteegen, E.
Tymchatyn, \emph{Fixed point theorems for plane continua with
applications}, Memoirs of the American Mathematical Society,
\textbf{224} (2013), no. 1053




\bibitem[BOPT14]{bopt14} A. Blokh, L. Oversteegen, R. Ptacek, V. Timorin,
\emph{The main cubioid}, Nonlinearity, \textbf{27} (2014), 1879-1897

\bibitem[BOPT14a]{bopt14a} A. Blokh, L. Oversteegen, R. Ptacek, V. Timorin,
\emph{Complementary components to the cubic Principal Hyperbolic
Domain}, preprint, arxiv:1411.2535 (2014)



\bibitem[BrHu]{BrHu2}
B. Branner, J. Hubbard, ``The iteration of cubic polynomials, Part
II: patterns and parapatterns'', Acta Mathematica \textbf{169}
(1992), 229--325.

\bibitem[BuHe01]{BH}
X. Buff, C. Henriksen, \emph{Julia Sets in Parameter Spaces},
Commun. Math. Phys. \textbf{220} (2001), 333 -- 375



\bibitem[DH8485]{DH}
A. Douady and J. Hubbard, \emph{\'Etude dynamique des polyn\^omes
complex I \& II} Publ. Math. Orsay (1984--85)

\bibitem[DH85]{DH85b}
A. Douady and J. Hubbard, \emph{On the dynamics of polynomial-like
mappings}, Ann. Sci. \'Ec. Norm. Sup. \textbf{18}, 287--343 (1985)





\bibitem[Hub93]{Hu}
J. H. Hubbard, ``Local connectivity of Julia sets and bifurcation
loci: three theorems of Yoccoz'', in: Topological Methods in Modern
Mathematics, Publish or Perish (1993)

\bibitem[Lev91]{lev91} G. Levin, \emph{On Pommerenke's inequality for the eigenvalues of
fixed points}, Colloq. Math. \textbf{62} (1991), 167–-177.

\bibitem[Lyu83]{lyu83} M. Lyubich, \emph{Some typical properties of the dynamics of
rational mappings}, Russian Math. Surveys \textbf{38} (1983), no. 5,
154--155.

\bibitem[Ma\~n85]{Ma}
R. Ma\~n\'e, \emph{Hyperbolicity, Sinks and Measure in One
Dimensional Dynamics}, Commun. Math. Phys. \textbf{100} (1985),
495--524.

\bibitem[Ma\~n93]{man93} R.~Ma\~{n}\'{e}, \emph{On a theorem of {F}atou},
Bol. Soc. Bras. Mat. \textbf{24} (1993), 1--11.

\bibitem[MSS83]{MSS}
R. Ma\~n\'e, P. Sad, D. Sullivan, \emph{On the dynamics of rational
maps}, Ann. Sci. \'Ecole Norm. Sup. (4) \textbf{16} (1983), no. 2,
193--217.

\bibitem[McM94]{Mc}
C. McMullen, \emph{Complex Dynamics and Renormalization}, Princeton
University Press, 1994.

\bibitem[McM00]{Mc00}
C. McMullen, \emph{The Mandelbrot set is universal}, In: The
Mandelbrot Set, Theme and Variations, Ed: Tan Lei.
London Mathematical Society Lecture Note Series, 274, 1--18.
Cambridge Univ. Press.

\bibitem[Pet93]{pet93} C. Petersen, \emph{On the Pommerenke-Levin-Yoccoz inequality},
Ergod. Th. and Dynam. Sys. \textbf{13} (1993), 785-–806.

\bibitem[PeTa09]{peta09} C. Petersen, Tan Lei, \emph{Analytic coordinates recording cubic dynamics},
in: Complex Dynamics, Families and Friends, ed: D. Schleicher, A K Peters (2009).

\bibitem[PiTa04]{pita04} K. Pilgrim, Tan Lei, \emph{Spinning deformations of rational maps},
Conformal Geometry and Dynamics \textbf{8} (2004), 52-–86.

\bibitem[Pom86]{pom86} Ch. Pommerenke, \emph{On conformal mapping and iteration of rational
functions}, Complex Variables Theory Appl. \textbf{5} (1986),
117–-126.











\bibitem[Zak99]{Z}
S. Zakeri, \emph{Dynamics of Cubic Siegel Polynomials}, Commun.
Math. Phys. \textbf{206} (1999), 185--233

\end{thebibliography}
\end{document}